\UseRawInputEncoding
\documentclass[final,12pt]{elsarticle}
\usepackage{ragged2e}
\justifying
\biboptions{sort&compress}
\usepackage{bm}
\usepackage{geometry}
\usepackage{amsmath}
\usepackage{mathrsfs}
\usepackage{amsthm}
\usepackage{amssymb}
\newtheorem{theorem}{Theorem}[section]

\newtheorem{lemma}{Lemma}[section]

\newtheorem{proposition}{Proposition}[section]

\newtheorem{remark}{Remark}
\numberwithin{equation}{section}

\newenvironment{proof of 1.7}{\noindent{\emph{Proof of Theorem 1.7 (1).}}\ }{\hfill $\square$\par}
\newenvironment{proof of 1.10}{\noindent{\emph{Proof of Theorem 1.10.}}\ }{\hfill $\square$\par}

\usepackage{indentfirst}
\begin{document}
	\begin{frontmatter}  
		\title{Spectral extrema of  $\{K_{k+1},\mathcal{L}_s\}$-free graphs \,\tnoteref{titlenote}}  
		\tnotetext[titlenote]{This work is supported by the National Natural Science Foundation  of China (Nos. 11871040, 12271337).}    
		\author{Yanni Zhai$^{1,2}$}  
		\author{Xiying Yuan$^{1,2}$\corref{correspondingauthor}}
	
		\cortext[correspondingauthor]{Corresponding author. \\
			Email address: yannizhai2022@163.com (Yanni Zhai), xiyingyuan@shu.edu.cn (Xiying Yuan).}   
		\address{$^1$Department of Mathematics, Shanghai University, Shanghai 200444, P.R. China} 
		\address{$^2$Newtouch Center for Mathematics of Shanghai University, Shanghai 200444, P.R. China} 
		\begin{abstract}  
		For a set of graphs $\mathcal{F}$, a graph is said to be $\mathcal{F}$-free  if it does not contain any graph in $\mathcal{F}$ as a subgraph.
		Let Ex$_{sp}(n,\mathcal{F})$ denote the graphs with the maximum spectral radius among all  $\mathcal{F}$-free graphs of order $n$. A linear forest is a graph whose connected component is a path. Denote by $\mathcal{L}_s$ the family of all linear forests with $s$ edges. In this  paper the graphs in Ex$_{sp}(n,\{K_{k+1},\mathcal{L}_s\})$ will be completely characterized when $n$ is appropriately large.
		\end{abstract}   
	\end{frontmatter}
	\section{Introduction}
	In this paper, the graph is considered to be simple and undirected.  The order of a graph $G=\left(V(G),\,E(G)\right)$ is the number of its vertices, and its size is the number of its edges, denoted by $e(G)$. Given two vertex-disjoint graphs $G$ and $H$, the union of graphs $G$ and $H$ is the graph $G\cup H$ with vertex set $V(G)\cup V(H)$ and edge set $E(G)\cup E(H)$.  The union of $k$ copies of $G$ is denoted by $kG$.
	The join of $G$ and $H$, denoted by $G\vee H$, is the graph obtained from $G\cup H$  by adding all edges between $V(G)$ and $V(H)$. 
	Let $K_{n_1,n_2,\cdots,n_k}$ be the complete $k$-partite graph with classes of sizes $n_1,\cdots ,n_k$. If $\sum_{i=1}^{k}n_i=n$ and $\lvert n_i-n_j\rvert \leq 1$ for any $1\leq i<j\leq k$, then $K_{n_1,n_2,\cdots,n_k}$ is called a $k$-partite Tur\'{a}n graph, denoted by $T_{n,\,k}$.
		A path of order $n$ is denoted  by $P_n$.
	Let $K_n$  denote the complete graph and $\overline{K}_n$ denote the complement graph of $K_n$.


	
    For a set  of graphs $\mathcal{F}$, define ex$(n,\mathcal{F})$  the maximum number of edges of an $\mathcal{F}$-free graph of order $n$ and Ex$(n,\mathcal{F})$ the set of $\mathcal{F}$-free graphs of order $n$ with ex$(n,\mathcal{F})$ edges. A fundamental theorem in extremal graph theory, Tur\'{a}n Theorem, gives Ex$(n,\{K_{k+1}\})=\{T_{n,\,k}\}$ for $n>k\geq 3$.  For a graph $G$, let $A(G)$ denote its adjacency matrix and let $\rho (G)$ denote the spectral radius of $G$. By Perron-Frobenius Theorem we have $\rho (G)$ is the largest eigenvalue of $A(G)$.
    Define  Ex$_{sp}(n,\mathcal{F})$ the set of $\mathcal{F}$-free graphs of order $n$ with maximum spectral radius. 
	Brualdi-Solheid-Tur\'{a}n type problem \cite{1986} is to characterize the set  Ex$_{sp}(n,\mathcal{F})$ for some given  $\mathcal{F}$. When $\mathcal{F}$ contains only one graph $F$, we write Ex$_{sp}(n,\mathcal{F})$ as Ex$_{sp}(n,F)$  for convenience.
	 In 2007, Nikiforov \cite{2007N} gave a spectral version of  Tur\'{a}n Theorem by showing that Ex$_{sp}(n,
	 K_{k+1})=\{T_{n,\,k}\}$ for $n>k\geq 3$.  
	 There have been fruitful results for the characterization of Ex$_{sp}(n,F)$.
	 In 2010, Nikiforov \cite{2010} determined Ex$_{sp}(n,F)$ when $F$ is a path or a cycle of specified length. 
	 When $F$ is a matching of given size, Ex$_{sp}(n,F)$ was determined by Feng, Yu and Zhang \cite{2007F}.
	 In 2019, Chen, Liu and Zhang \cite{2019C} considered Ex$_{sp}(n,F)$ when $F$ is a specific linear forest. In 2020, Cioab\v{a}, Feng, Tait and Zhang \cite{2020} studied Ex$_{sp}(n,F)$ when $F$ is a friendship. 
	  Cioab\v{a}, Desai and Tait \cite{2022C} characterized the graphs in Ex$_{sp}(n,F)$ when $F$ is an odd wheel.
	 For other classes of graphs, readers may be referred to \cite{2009B,2022C1,2019G,2023Ni,2023WJ}.
	Recently,  Brualdi-Solheid-Tur\'{a}n type problem has been considered when $\mathcal{F}$ contains more than one graph.
	  Write $M_s$ for a matching  of $s$ edges.  Wang, Hou and Ma \cite{2023WH} considered Ex$_{sp}(n,\mathcal{F})$ for $\mathcal{F}=\{K_{k+1},M_{s+1}\}$.


	 \begin{theorem}[{Wang, Hou and Ma~\cite[]{2023WH}}]\label{Wang}
	 	Suppose  $k\geq 2$ and  $n\geq 4s^2+9s$. Then
	 	\begin{equation*}
	 		T_{s,\,k-1}\vee \overline{K}_{n-s}\in {\rm{Ex}}_{sp}(n,\{K_{k+1},M_{s+1}\}).
	 	\end{equation*}
	 \end{theorem}
	 Note that $M_s\in \mathcal{L}_s$ holds.  As an extension,
	 in this paper we consider the characterization of Ex$_{sp}(n,\{K_{k+1},\mathcal{L}_s\})$
	   when $n$ is appropriately large. 
	   When Ex$_{sp}(n,\mathcal{F})=\{G\}$, write  Ex$_{sp}(n,\mathcal{F})=G$ for convenience.
	\begin{theorem}\label{main}
		Suppose $k\geq 2$, $s\ge 2$ and $n\geq 36288\left(\lfloor\frac{s+1}{2}\rfloor\right)^8$.
		\begin{enumerate}[(1)]
			\item When $s$ is odd and $s\leq 2k-1$, {\rm{Ex}}$_{sp}(n,\{K_{k+1},\mathcal{L}_{s}\})=K_{\frac{s-1}{2}}\vee \overline{K}_{n-\frac{s-1}{2}}$ holds.\\
			 When $s$ is odd and $s\geq 2k+1$, {\rm{Ex}}$_{sp}(n,\{K_{k+1},\mathcal{L}_{s}\})=T_{\frac{s-1}{2},\,k-1}\vee \overline{K}_{n-\frac{s-1}{2}}$ holds.
			\item When $s$ is even and $s\leq 2k-2$, {\rm{Ex}}$_{sp}(n,\{K_{k+1},\mathcal{L}_{s}\})=K_{\frac{s}{2}-1}\vee \left(P_2\cup \overline{K}_{n-\frac{s}{2}-1}\right)$ holds.
			 When  $s$ is even and $s\geq 2k$, {\rm{Ex}}$_{sp}(n,\{K_{k+1},\mathcal{L}_{s}\})=T_{\frac{s}{2}-1,\,k-1}\vee \overline{K}_{n-\frac{s}{2}+1}$ holds.
		\end{enumerate}
	\end{theorem}
	
When $s=2$, we have Ex$_{sp}(n,\{K_{k+1},\mathcal{L}_{2}\})=P_2\cup \overline{K}_{n-2}$ and then the conclusion holds. Therefore, we next assume $s\geq 3$.

\section{Some characterizations of the graphs in Ex$_{sp}\left(n,\{K_{k+1},\mathcal{L}_{s}\}\right)$}
 For a vertex $v\in V(G)$, the neighborhood of $v$ in $G$ is denoted by $N_{G}(v)$ or simply $N(v)$. Set  $N[v]=N(v)\cup \{v\}$.  Denote $N_G^2(v)$ or simply $N^2(v)$ by the set of vertices at distance two from $v$ in $G$.  	For $U\subseteq V(G)$, let $G[U]$ be the subgraph of $G$ induced by $U$. For $V_1,\,V_2\subseteq V(G)$ and $V_1\cap V_2=\emptyset$,  $G[V_1,\,V_2]$ denotes the induced bipartite graph with one partite set $V_1$ and the other partite set $V_2$.
 
 We always suppose
 $k\geq 2$, $s\geq 3$, $h=\lfloor\frac{s-1}{2}\rfloor$, $\alpha =\frac{1}{36(h+1)^3}$ and $n\geq \frac{28(h+1)^2}{\alpha ^2}$. Then 
 $n\geq 36288\left(\lfloor \frac{s+1}{2}\rfloor \right)^8$ holds.
 Let $G$ be a graph in Ex$_{sp}\left(n,\{K_{k+1},\mathcal{L}_{s}\}\right)$.
Lemma \ref{connected} shows that $G$ is connected. Furthermore, we may prove that there exists a vertex subset $R\subseteq V(G)$ containing vertices with larger degree, namely, for each vertex $v$ in $R$, $d_G(v)>\left(1-\frac{5}{6(h+1)}\right)n$ holds. We will show $\lvert R\rvert =h$. Then set $$W=\{v\in V(G):R\subseteq N_G(v)\},$$ which means $G[R,W]$ is a complete bipartite graph. The fact $d_G(v)>\left(1-\frac{5}{6(h+1)}\right)n$  implies 
\begin{equation*}
	\lvert W\rvert \geq n-\frac{5n\lvert R\rvert }{6(h+1)}>s+1.
\end{equation*}
  For any graph $G$ in Ex$_{sp}(n,\{K_{k+1},\mathcal{L}_s\})$, the fact that $K_{h,\,n-h}$ is $\{K_{k+1},\mathcal{L}_s\}$-free implies a natural lower bound on $\rho (G)$, 
\begin{equation}\label{lower}
	\rho(G)\geq \rho(K_{h,\,n-h})=\sqrt{h(n-h)}.
\end{equation}
	
	In 2020, Ning and Wang gave the Tur\'{a}n number of the linear forest $\mathcal{L}_s$.
	\begin{theorem}[{Ning and Wang~\cite[]{2020N}}]\label{Ning}
		For any integers $n$ and $s$ with $1\leq s\leq n-1$, we have 
		\begin{equation*}
			{\rm{ex}}(n,\mathcal{L}_{s})={\rm{max}}\,\bigg\{\binom{s}{2},\,\binom{n}{2}-\binom{n-\lfloor\frac{s-1}{2}\rfloor}{2}+c\bigg\},
		\end{equation*}
		where $c=0$ if $k$ is odd and $c=1$ otherwise.
	\end{theorem}
	\begin{remark}
		Based on the result of Theorem \ref{Ning}, 
		  when $s\geq 3$, if $G$ is in {\rm{Ex}}$_{sp}(n,\{K_{k+1},\mathcal{L}_s\})$, then we have
			\begin{equation}\label{upper}
				e(G)\leq 	{\rm{ex}}(n,\mathcal{L}_{s})\leq \lfloor\frac{s-1}{2}\rfloor n.
			\end{equation}
	When $s$ is odd, we have $K_{\frac{s-1}{2}}\vee \overline{K}_{n-\frac{s-1}{2}}\in \,${\rm{Ex}}$(n,\mathcal{L}_s)$. When $s$ is even, we have $K_{\frac{s-2}{2}}\vee (P_2\cup \overline{K}_{n-\frac{s}{2}-1})\in \,${\rm{Ex}}$(n,\mathcal{L}_s)$.
	\end{remark}

	The following formulas are introduced in \cite{arxiv2}. For a graph $H$ and any two vertex subsets  $A$ and $B$ of $V(H)$, $e(A,\,B)$ denotes the number of the edges of $H$ with one end vertex in $A$ and the other in $B$. 
	\begin{equation}\label{,1}
		e(A,\,B)=e(A,\,B\setminus A)+e(A,\,A\cap B)=e(A,\,B\setminus A)+2e(H[A\cap B])+e(A\setminus B,\, A\cap B),
	\end{equation}
	\begin{equation}\label{,2}
		e(A,\,B)\leq e(H[A\cup B])+e(H[A\cap B])\leq 2e(H),
	\end{equation}
	\begin{equation}\label{,3}
		e(A,\,B)\leq \lvert A\rvert \lvert B\rvert .
	\end{equation}

	\begin{lemma}\label{connected}
		 If $G$ is in {\rm{Ex}}$_{sp}\left(n,\{K_{k+1},\mathcal{L}_{s}\}\right)$, then 
		 $G$ is connected.
	\end{lemma}
\begin{proof}
	 Suppose to the contrary that  $G_1$ is a component of $G$ with 
	$\rho(G_1)=\rho(G)\geq \sqrt{h(n-h)}$ and then $\lvert V(G_1)\rvert\leq n-1$. Let $u\in V(G_1)$ with the maximum  degree in $G_1$, $G'$ be the graph obtained from $G_1$ by adding a pendent edge $uv$ at $u$ and $n-\lvert V(G_1)\rvert -1$ isolated vertices.  If there is a copy  $K_{k+1}$ in $G'$, then there is a copy  $K_{k+1}$ in $G_1$, a contradiction. If there is a copy $L$ of $\mathcal{L}_{s}$ in $G'$, then $uv\in E(L)$. As $d_{G_1}(u)\geq \rho (G_1)\geq \sqrt{h(n-h)}>s$, we may find a vertex $w\in N_{G_1}(u)\setminus V(L)$. Replace $uv$ with $uw$, and then there is a copy  $L$ in $G$, a contradiction.
	Therefore $G'$ is $\{K_{k+1},\mathcal{L}_s\}$-free.
	  While by Perron-Frobenius Theorem, we have
	  $\rho(G')>\rho (G_1)=\rho(G)$. 
	\end{proof}
	  Let \textbf{x} be the Perron vector of $G$, i.e., $A(G)\textbf{x}=\rho (G)\textbf{x}$ and $\Vert \textbf{x}\Vert _2=1$. Since $G$ is connected, by Perron-Frobenius Theorem   $\textbf{x}$ is a positive vector. 
	  Set $x_z=$max$\{x_v:\,v\in V(G)\}$.
	  Write $\rho=\rho (G)$ for convenience.  To show the existence of the set $R$, we need two auxiliary sets $R'$ and $R''$ in terms of the components of $\textbf{x}$. 
	  With extensive use of the eigenequations of $A$ and $A^2$, some  basic structural properties of the spectral extremal graphs can be deduced.
	  Cioab\v{a}, Desai and Tait \cite{2022C1} used this technique to  present the spectral  Erd\H{o}s-S\'{o}s Theorem. Recently, Zhai and Liu \cite{arxiv2} proved the spectral Erd\H{o}s-P\'{o}sa Theorem by using this technique.

	  Let $$R'=\{v\in V(G):x_v>\alpha x_z\},$$ and $\overline{R'}=V(G)\setminus R'$. The following lemma gives an upper bound of the cardinality of $R'$.
	  
	  \begin{lemma}\label{l1}
	  	$\lvert R'\rvert \leq 2\sqrt{hn}$.
	  \end{lemma}
  \begin{proof}
  	For each vertex $v\in R'$, we have 
  	\begin{align}\label{R}
  		\sqrt{h(n-h)}\alpha x_z&\leq \rho x_v=\sum_{u\in N(v)}x_u\leq \sum_{u\in (N(v)\cap R')}x_z+\sum_{u\in (N(v)\setminus R')}\alpha x_z \notag \\ &=d_{G[R']}(v)x_z+d_{G[\overline{R'}]}(v)\alpha x_z.
  	\end{align}
  Summing the inequality (\ref{R}) for all $v\in R'$, we have
  \begin{equation*}
  	\lvert R'\rvert \sqrt{h(n-h)}\alpha x_z\leq \sum_{v\in R'}d_{G[R']}(v)x_z+\sum_{v\in R'}d_{G[\overline{R'}]}(v)\alpha x_z
  	=2e(G[R'])x_z+e(R',\,\overline{R'})\alpha x_z. 
  \end{equation*}
If $\lvert R'\rvert \leq s$, the conclusion holds. Now suppose $\lvert R'\rvert \geq s+1$. Since $G$ is $\mathcal{L}_s$-free, by (\ref{upper}), we have $e(G[R'])\leq h\lvert R'\rvert $ and $e(R',\overline{R'})\leq hn$. Hence we obtain 
\begin{equation*}
	 \sqrt{h(n-h)} \alpha \lvert R'\rvert \leq 2h\lvert R'\rvert+\alpha hn.
\end{equation*} 
Since $n\geq \frac{28(h+1)^2}{\alpha ^2}$, $\alpha \sqrt{h(n-h)}-2h\geq \frac{1}{2}\alpha \sqrt{hn}$ holds, then we obtain
\begin{equation*}
	\lvert R'\rvert \leq \frac{\alpha hn}{\alpha \sqrt{h(n-h)}-2h}\leq \frac{\alpha hn}{\frac{1}{2}\alpha \sqrt{hn}}=2\sqrt{hn}.
\end{equation*}
  \end{proof}

Set $R''=\{v\in V(G):x_v>4\alpha x_z\}$.
\begin{lemma}\label{l2}
	For each vertex $v\in R''$, $d_G(v)>\frac{1}{3}\alpha n$ holds.
\end{lemma}
\begin{proof}
	To prove the lemma, we may prove the following claim first.
	
		\textbf{Claim.}
		For any vertex $v$ in $R'$, if $d_G(v)\leq \frac{1}{3}\alpha n$, then $e\left(G[(N(v)\cup R')\setminus \{v\}]\right)\leq \frac{5h-1}{3}\alpha n$ holds.
		
		From $$4\sqrt{h(n-h)}\alpha x_z\leq \rho x_v=\sum_{u\in N(v)}x_u\leq d_G(v)x_z,$$
		we have $d_G(v)\geq 4\sqrt{h(n-h)}\alpha \geq s+1$.
		Since $G[N(v)\cup R']$ is $\mathcal{L}_s$-free, by (\ref{upper}), we have
	\begin{equation*}
		e(G[N(v)\cup R'])\leq h(d_G(v)+\lvert R'\rvert ),
	\end{equation*}
which implies
\begin{align*}
	e(G[(N(v)\cup R')\setminus\{v\}])&\leq h(d_G(v)+\lvert R'\rvert )-d_G(v)\\
	&=(h-1)d_G(v)+h\lvert R'\rvert \\
	&\leq \frac{1}{3}\alpha (h-1)n+2h\sqrt{hn}\\
	&\leq \frac{5h-1}{3}\alpha n,
\end{align*}
where the last inequality follows from $n\geq\frac{28(h+1)^2}{\alpha ^2}\geq \frac{9h}{4\alpha ^2} $.

Now we prove Lemma \ref{l2}. Suppose to the contrary that there is a vertex $v\in R''$ with $d_G(v)\leq \frac{1}{3}\alpha n$. By the claim above and the fact 
\begin{equation*}
	\sum_{u\in (N^2(v)\cap R')}d_{G[N(v)]}(u)= e(N(v),\,N^2(v)\cap R')\leq e(G[(N(v)\cup R')\setminus\{v\}]),
\end{equation*}
we obtain 
\begin{equation}\label{2.1}
	\sum_{u\in (N^2(v)\cap R')}d_{G[N(v)]}(u)x_u\leq e(N(v),\,N^2(v)\cap R')x_z\leq \frac{(5h-1)\alpha nx_z}{3}.
\end{equation} 
As $G[N(v)]$ is $\mathcal{L}_s$-free, we obtain $e(G[N(v)])\leq hd_G(v)$. It follows 
\begin{equation}\label{2.2}
	\left(d_G(v)+2e(G[N(v)])\right)x_z\leq (2h+1)d_G(v)x_z\leq \frac{(2h+1)\alpha n x_z}{3}.
\end{equation}
Moreover, the fact $e(N(v),N^2(v)\setminus R')\leq hn$ implies 
\begin{equation}\label{2.3}
	\sum_{u\in (N^2(v)\setminus R')}d_{G[N(v)]}(u)x_u\leq e(N(v),N^2(v)\setminus R')\alpha x_z\leq \alpha hn x_z.
\end{equation}
Combining (\ref{2.1}) - (\ref{2.3}), we have
\begin{equation*}
	\begin{split}
		\rho ^2x_v&=d_G(v)x_v+\sum_{u\in N(v)}d_{G[N(v)]}(u)x_u+\sum_{u\in N^2(v)}d_{G[N(v)]}(u)x_u\\
	&\leq \big(d_{G}(v) +2e(G[N(v)])\big)x_z +\sum_{u\in (N^2(v)\cap R')}d_{G[N(v)]}(u)x_u+\sum_{u\in (N^2(v)\setminus R')}d_{G[N(v)]}(u)x_u \\
	&\leq \frac{(2h+1)\alpha n x_z}{3} +\frac{(5h-1)\alpha nx_z}{3}+\alpha hn x_z.
	\end{split}
\end{equation*}
On the other hand, by (\ref{lower}), we have
\begin{equation*}
	\rho ^2x_v\geq h(n-h)x_v\geq 4\alpha h(n-h)x_z.
\end{equation*}
Then we obtain the following inequality
\begin{equation*}
	4\alpha h(n-h)x_z\leq \frac{(2h+1)\alpha n x_z}{3} +\frac{(5h-1)\alpha nx_z}{3}+\alpha hn x_z.
\end{equation*}
Simplifying the inequality, we obtain $n\leq 6h$, which is a contradiction. Therefore, each vertex in $R'$ has degree larger than $\frac{1}{3}\alpha n$.
\end{proof}

\begin{lemma}\label{l3}
	$\lvert R''\rvert < \frac{3(h+1)}{\alpha}$.
\end{lemma}
\begin{proof}
	If $\lvert R''\rvert \leq s$, the conclusion holds. Now suppose $\lvert R''\rvert \geq s+1$. By (\ref{upper}), we have $e(G[R''])\leq h\lvert R''\rvert $. From Lemma \ref{l2}, $\sum_{v\in R''}d_G(v)>\frac{1}{3}\alpha n \lvert R''\rvert$ holds, and we get
	\begin{equation*}
		\sum_{v\in \overline{R''}}d_{G}(v)\geq e(\overline{R''},\, R'')=\sum_{v\in R''}d_{G}(v)-2e(G[R''])> \frac{1}{3}\alpha n\lvert R''\rvert -2h\lvert R''\rvert.
	\end{equation*}
Therefore, we prove the lemma
\begin{equation*}
	\begin{split}
		hn\geq e(G)&=\frac{1}{2}\sum_{v\in R''}d_G(v)+\frac{1}{2}\sum_{v\in \overline{R''}}d_G(v)\\
		&>\frac{1}{6}\alpha n \lvert R''\rvert +\frac{1}{6}\alpha n \lvert R''\rvert-h\lvert R''\rvert\\
		&=\frac{1}{3}\alpha n\lvert R''\rvert -h\lvert R''\rvert.
	\end{split}
\end{equation*}
As $n\geq \frac{28(h+1)^2}{\alpha ^2}\geq \frac{3h(h+1)}{\alpha}$, we have $\lvert R''\rvert<\frac{hn}{\frac{1}{3}\alpha n-h}\leq \frac{3(h+1)}{\alpha}.$
\end{proof}

\begin{lemma}\label{l4}
	If $v$ is a vertex with $x_v=mx_z$ and $\frac{1}{2(h+1)}\leq m\leq 1$, then $d_G(v)>\left(m-\frac{1}{6(h+1)}\right)n$ holds.
\end{lemma}

\begin{proof}
		Suppose to the contrary that there is a vertex $v$ with $x_v=mx_z$ $\left(\frac{1}{2(h+1)}\leq m\leq 1\right)$ and $d_{G}(v)\leq \left(m-\frac{1}{6(h+1)}\right)n.$ By the definition of $R''$, $v\in R''$ holds.
		Let $M=N(v)\cup N^2(v)$. From (\ref{upper}), we obtain 
		\begin{equation}\label{4.1}
			e(M\setminus R'',N(v))\leq 2e(G)\leq 2hn.
		\end{equation}
	Since $v\in R''$, we have $d_G(v)>\frac{1}{3}\alpha n$  from Lemma \ref{l2}. Then $G[R''\setminus \{v\},\,N(v)\setminus R'']$ is $\mathcal{L}_{s-2}$-free. Otherwise there is a copy of $\mathcal{L}_s$ in $G$ as the degree of $v$ in $G$ is large enough. Hence, Theorem \ref{Ning} implies 
	\begin{equation*}
		e\left(R''\setminus \{v\},\,N(v)\setminus R''\right)\leq (h-1)(d_G(v)+\lvert R''\rvert -1)+1.
	\end{equation*}
Then from (\ref{upper}) and (\ref{,1}), we have
\begin{align}\label{4.2}
	e\left(M\cap R'',N(v)\right)&\leq e(R''\setminus \{v\},\, N(v)) \notag \\
	&= e(R''\setminus \{v\},\,N(v)\setminus R'')+2e\left(G[R''\cap N(v)]\right)+e\left(R''\setminus N[v],\, R''\cap N(v)\right) \notag \\
	&\leq (h-1)(d_G(v)+\lvert R''\rvert -1)+2e(G[R''])+1.
\end{align} 

If $\lvert R''\rvert \geq s+1$, then by (\ref{upper}), $e(G[R''])\leq h\lvert R''\rvert$ holds. From Lemma \ref{l3}, $\lvert R''\rvert <\frac{3(h+1)}{\alpha}$ holds. We have 
\begin{equation}\label{4.3}
	(h-1)\lvert R''\rvert +2e(G[R''])\leq (h-1)\lvert R''\rvert +2h\lvert R''\rvert<\frac{(9h-3)(h+1)}{\alpha}.
\end{equation}

If $\lvert R''\rvert \leq s$, then $e(G[R''])<\frac{s(s-1)}{2}$ holds. We have
\begin{equation}\label{4.4}
	(h-1)\lvert R''\rvert +2e(G[R''])\leq  (h-1)s+s(s-1)< \frac{(9h-3)(h+1)}{\alpha},
\end{equation}
the last inequality holding as $\alpha =\frac{1}{36(h+1)^3}$.

Combining (\ref{4.1}) - (\ref{4.4}), we obtain
\begin{align*}
	\rho ^2x_v&=d_{G}(v)x_v+\sum_{u\in (M\setminus R'')}d_{G[N(v)]}(u)x_u+\sum_{u\in (M\cap R'')}d_{G[N(v)]}(u)x_u \notag\\
	&\leq d_{G}(v)x_v+e(M\setminus R'', \,N(v))4\alpha x_z +e(M\cap R'', \,N(v))x_z\\
		&< d_{G}(v)mx_z+8\alpha hnx_z+(h-1)d_{G}(v)x_z+\frac{(9h-3)(h+1)}{\alpha}x_z-(h-2)x_z.
\end{align*}
On the other hand, by (\ref{lower}), $\rho^2 x_v\geq h(n-h)mx_z$ holds. Therefore, we have
\begin{equation*}
	hnm-h^2m< d_G(v)m+8\alpha hn+(h-1)d_G(v)+\frac{(9h-3)(h+1)}{\alpha}-(h-2).
\end{equation*}
Since $d_G(v)\leq \left(m-\frac{1}{6(h+1)}\right)n$, we have 
\begin{equation*}
	-nm^2+\left(-h^2+\frac{6h+7}{6(h+1)}n\right)m+\frac{(h-1)}{6(h+1)}n-8\alpha hn< \frac{(9h-3)(h+1)}{\alpha}-(h-2).
\end{equation*}
Let $$g(x)=-nx^2+\left(-h^2+\frac{6h+7}{6(h+1)}n\right)x+\frac{(h-1)}{6(h+1)}n-8\alpha hn.$$ Then $g(x)\geq g(1)$ holds when $x\in \left[\frac{1}{2(h+1)},1\right]$. So $g(m)\geq g(1)$ holds and we have
\begin{equation*}
	\left(\frac{h}{6(h+1)}-8\alpha h\right)n< \frac{(9h-3)(h+1)}{\alpha }+h^2-(h-2).
\end{equation*}
As $\frac{h}{6(h+1)}-8\alpha h >\frac{\alpha h}{6}$, we have $n<\frac{54(h+1)+6\alpha h}{\alpha ^2}$. While it contradicts the fact $n\geq \frac{28(h+1)^2}{\alpha ^2}$.
\end{proof}

Let $R=\{v\in V(G):x_v\geq \frac{1}{2(h+1)}x_z\}$. Clearly $R\subseteq R''$ holds. We will prove that $R$ is the desired set. Set $N=N(z)\cup N^2(z)$.
\begin{lemma}\label{l5}
	$ d_G(v)> \left(1-\frac{5}{6(h+1)}\right)n$ holds for each $v$ in $R$.
\end{lemma}
\begin{proof}
	By Lemma \ref{l4}, it suffices to show that $x_v\geq \left(1-\frac{2}{3(h+1)}\right)x_z$ holds for each $v\in R$. Suppose to the contrary that there exists a vertex $v\in R$ such that $x_v=mx_z$, where $\frac{1}{2(h+1)}\leq m<1-\frac{2}{3(h+1)}$.
	Since $z,\,v\in R$, by Lemma \ref{l4} we have $d_G(v)>\left(m-\frac{n}{6(h+1)}\right)\geq \frac{n}{3(h+1)}$ and $d_G(z)>\left(1-\frac{1}{6(h+1)}\right)n$. Then we have 
	$d_{G[N(z)]}(v)> \frac{n}{6(h+1)}$. Hence $v\in (N\cap R'')$ holds.
	 On the one hand,  we have $d_G(z)+e(N\cap R'',N(z))\leq e(R'',N(z))$. On the other hand,  the fact $R''\cup (N\setminus R'')\subseteq V(G)$ and (\ref{upper}) imply
	\begin{equation}\label{5.1}
		e(R'',N(z))+e(N\setminus R'',N(z))\leq e(V(G),N(z))\leq 2hn.
	\end{equation}
Then we have
\begin{align}\label{5.2}
	&d_G(z)+e(N\cap R'',N(z))+e(N\setminus R'',N(z))4\alpha \notag\\
	\leq& e\left(R'',N(z)\right)+\left(2hn-e(R'',N(z))\right) 4\alpha\notag\\
	=&(1-4\alpha )e(R'',N(z))+8\alpha hn.
\end{align}
As $G[R'']$ is $\mathcal{L}_s$-free, $e(G[R''])\leq $ max $\{h\lvert R''\rvert,\,\binom{s}{2}\}< \frac{3h(h+1)}{\alpha }$. By   (\ref{upper}) and (\ref{,2}), we have 
\begin{equation}\label{5.3}
	e(R'',N(z))\leq e(R''\cup N(z))+e(R''\cap N(z))\leq e(G)+e(G[R''])< hn+\frac{3h(h+1)}{\alpha }.
\end{equation}
Combining (\ref{5.2}) and (\ref{5.3}), we may see that
\begin{align*}
	\rho ^2x_z&=d_G(z)x_z+\sum_{u\in (N\cap R'')}d_{G[N(z)]}(u)x_u+\sum_{u\in (N\setminus R'')}d_{G[N(z)]}(u)x_u\\
	&\leq \left(d_G(z)+e\left(N\cap R'',N(z)\right)\right)x_z+d_{G[N(z)]}(v)(x_v-x_z)+e(N\setminus R'',N(z))4\alpha x_z\\
	&\leq (1-4\alpha )e(R'',N(z))x_z+8\alpha hnx_z+d_{G[N(z)]}(v)(x_v-x_z) \\
	&<\left((1-4\alpha )\left(hn+\frac{3h(h+1)}{\alpha }\right)+8\alpha hn\right)x_z+d_{G[N(z)]}(v)(x_v-x_z)\\
	&=\left(hn+\frac{3h(h+1)}{\alpha}-12h(h+1)+4\alpha hn\right)x_z+d_{G[N(z)]}(v)(m-1)x_z.
\end{align*}
On the other hand $\rho ^2\geq h(n-h)$ and $m<1-\frac{2}{3(h+1)}$ hold, we may get 
\begin{equation}\label{5.4}
	\frac{2}{3(h+1)}d_{G[N(z)]}(v)< h^2+\frac{3h(h+1)}{\alpha}-12h(h+1)+4\alpha hn.
\end{equation}
It follows from (\ref{5.4}) that
\begin{equation*}
	\frac{n}{9(h+1)^2}-4\alpha hn<h^2+\frac{3h(h+1)}{\alpha}-12h(h+1),
\end{equation*}
which implies $n<\frac{(h+1)^2}{\alpha ^2}$, a contradiction.
\end{proof}

Now we prove the exact value of $\lvert R\rvert $. To do it, we need the following inequality presented by \cite{2020}.
\begin{lemma}[{Cioab\v{a}, Feng, Tait and Zhang~\cite[]{2020}}]\label{cup}
	Suppose  $S_1,\,\cdots,\,S_{k}$ are $k$ finite sets. Then 
	\begin{equation*}
		\lvert S_{1}\cap \cdots \cap S_{k}\rvert \geq \sum_{i=1}^{k}\lvert S_{i}\rvert -(k-1)\lvert \cup^{k} _{i=1}S_i\rvert.
	\end{equation*}
\end{lemma}

\begin{lemma}\label{l6}
	$\lvert R\rvert=h$. 
\end{lemma}
\begin{proof}
	Firstly, we prove $\lvert R\rvert>h-1$. Suppose to the contrary $\lvert R\rvert\leq h-1$.  By (\ref{upper}), we have 
	\begin{equation}\label{6.1}
		e(N\setminus R,\,N(z))\leq 2e(G)\leq 2hn.
	\end{equation}
The fact $N\cap R \subseteq R\setminus\{z\}$ and (\ref{,3}) imply
\begin{equation}\label{6.2}
	e(N\cap R,\,N(z))\leq \lvert N\cap R\rvert\cdot\lvert N(z)\rvert\leq \lvert R\setminus \{z\}\rvert\, d_G(z).
\end{equation}

Combining (\ref{6.1}) and (\ref{6.2}), we have
\begin{align}\label{2.15}
	\rho ^2x_z&=d_G(z)x_z+\sum_{u\in (N\cap R)}d_{G[N(z)]}(u)x_u+\sum_{u\in (N\setminus R)}d_{G[N(z)]}(u)x_u \notag \notag \\
	&\leq d_G(z)x_z+e(N\cap R,\,N(z))x_z+e(N\setminus R,\,N(z))\frac{1}{2(h+1)}x_z\notag \\
	&\leq d_G(z)x_z+\lvert R\setminus \{z\}\rvert d_G(z) x_z+\frac{2hn}{2(h+1)}x_z \notag \\
	&=\lvert R\rvert d_G(z)x_z +\frac{hn}{h+1}x_z\notag.
\end{align}
Together with the fact $\rho ^2\geq h(n-h)$, it follows
\begin{equation*}
	h(n-h)\leq (h-1)(n-1)+\frac{hn}{h+1},
\end{equation*}
which implies  $n\leq (h+1)(h^2-h+1)$, a contradiction. Hence we have $\lvert R\rvert >h-1$.

If $\lvert R\rvert \geq h+1$, then we suppose $\{v_1,\,v_2,\cdots,v_{h+1}\}\subseteq R$ and $Q=N(v_1)\cap \cdots \cap N(v_{h+1})$. By Lemma \ref{cup}, it follows
\begin{equation*}
	\begin{split}
		\lvert Q\rvert  &\geq \sum_{i=1}^{h+1}d_G(v_i)-hn > (h+1)\left(1-\frac{5}{6(h+1)}\right)n-hn\geq 2(h+1),
	\end{split}
\end{equation*}
the last inequality holding as $n\geq \frac{28(h+1)^2}{\alpha  ^2}$. Then we may find $(h+1)P_3$ in $G[R,Q]$ which implies $\mathcal{L}_s\subseteq G$, a contradiction.
Therefore, $\lvert R\rvert =h $ holds.
\end{proof}

The results of Lemma \ref{l5} and Lemma \ref{l6} ensure the existence of the set $R$. And then we may set $W$ as 
$W=\{v\in V(G):R\subseteq N_G(v)\}$. Then $\lvert W\rvert >s+1$ holds and $G[R,W]$ is a complete bipartite graph. In Section 3 we will prove $G[R]$ is a complete multipartite graph and $W=\overline{R}$. By further detailed discussions, the extremal graph $G$ will be completely characterized.

\section{Proof of Theorem \ref{main}}	
For two non-adjacent vertices $u,v$ in a graph $H$, denote the Zykov symmetrization $Z_{u,v}(H)$ to be the graph obtained from $H$ by deleting all edges incident to vertex $u$ and then adding new edges from $u$ to $N_H(v)$ (see \cite{1949Z}).
Zykov symmetrization operation does not increase the size of the largest cliques, and it may strictly increase the spectral radius (see Proposition \ref{p1}).

 \begin{proposition}[{Li and Peng~\cite[]{2022L}}]\label{p1}
 	Suppose $H$ is a connected graph and $u,v$ are two non-adjacent vertices of $H$ with $N_H(u)\neq N_H(v)$. If $\sum_{w\in N_{H}(v)}x_w\geq \sum_{w\in N_{H}(u)}x_w$, then $\rho (Z_{u,v}(H))>\rho (H)$ holds.
\end{proposition}


The following lemma shows that for a complete multipartite graph the more balanced the graph is, the larger the spectral radius will be.
\begin{lemma}[{Wang, Hou and Ma~\cite[]{2023WH}}]\label{balance}
	For any complete $k$-partite graph $K_{n_1,\cdots,n_k}$ of order $n$, if there exist $i$ and $j$ with $n_i-n_j\geq 2$, then $\rho (K_{n_1,\cdots,n_{i}-1,\cdots, n_{j}+1,\cdots, n_k})>\rho (K_{n_1,\cdots,n_i,\cdots,n_j,\cdots,n_k})$ holds.
\end{lemma}

\vspace{0.1em}\noindent
\textbf{Proof of Theorem \ref{main} } 
Let $G$ be a graph in Ex$_{sp}(n,\{K_{k+1},\mathcal{L}_s\})$ and $n\geq 36288\left(\lfloor\frac{s+1}{2}\rfloor\right)^8$. We have proved that $G$ is connected. There is a vertex subset $R\subseteq V(G)$ such that $d_G(v)> \left(1-\frac{5}{6\lfloor\frac{s+1}{2}\rfloor} \right)n $ holds for each vertex $v\in R$ and $\lvert R\rvert =\lfloor\frac{s-1}{2}\rfloor$. Moreover, there is a vertex subset $W\subseteq V(G)$ such that $G[R,W]$ is a complete bipartite graph and $\lvert W\rvert >s+1$. 

\vspace{0.8em}\noindent
\textbf{Claim.}
$G[R]$ is a complete multipartite  graph.

In fact we may prove that
every two non-adjacent vertices of $R$ have the same neighborhood in $G$.
Suppose to the contrary that there exist two non-adjacent vertices $u,v\in R$ with $N_G(u)\neq N_G(v)$. Without loss of generality, suppose $\sum_{w\in N_G(v)}x_w\geq \sum_{w\in N_G(u)}x_w$.  By Proposition \ref{p1}, we have $\rho(Z_{u,v}(G))> \rho(G)$. And $Z_{u,v}(G)$ is $K_{k+1}$-free. If $Z_{u,v}(G)$ contains a copy $L\in \mathcal{L}_s$, then $u\in V(L)$ and $G[V(G)\setminus\{u\}]$ contains a copy of $L'\in \mathcal{L}_{s-2}$. Since $d_G(u)>\left(1-\frac{5}{6\lfloor\frac{s+1}{2}\rfloor} \right)n>s$, there is a $P_3=u_1uu_2$ with $\{u_1,u_2\}\subseteq V(G)\setminus V(L')$. Then $G$ contains a linear forest with $s $ edges, a contradiction. Hence $Z_{u,v}(G)$ is $\{K_{k+1},\mathcal{L}_s\}$-free. This is a contradiction. So every two non-adjacent vertices of $R$ have the same neighborhood in $G$. Therefore, $G[R]$ is a complete multipartite  graph. 

 Suppose $G[R] $ is a complete $\ell$-partite graph. Since $G$ is $K_{k+1}$-free and $\lvert R\rvert =\lfloor\frac{s-1}{2}\rfloor $, we have $\ell \leq$ min $\{\lfloor\frac{s-1}{2}\rfloor,k-1\}$. Let its partition sets be $B_1,\cdots,B_{\ell}$, $\lvert B_i\rvert=b_i $ and  $\sum_{i=1}^{\ell }b_i=\lfloor\frac{s-1}{2}\rfloor $.   If there is a vertex in $B_i$ adjacent to $v\in \overline{R}$, then we have $B_i\subseteq N_G(v)$ as every two non-adjacent vertices of $R$ have the same neighborhood in $G$.
 Now we distinguish two cases according to the parity of $s$.
 
~\\

\textbf{Case 1.} $s$ is odd.

In this case, we will prove $G[\overline{R}]=\overline{K}_{n-\frac{s-1}{2}}$. Suppose to the contrary that there is an edge $uv$ in $G[\overline{R}]$. Since $\lvert W\rvert >s+1$ and $G[R,W]$ is a complete bipartite graph, $G[(R\cup W)\setminus\{u,v\}]$ contains a copy of $\frac{s-1}{2}P_{3}$. Then $G$ contains a copy of $P_2\cup \frac{s-1}{2} P_3$, which contradicts the assumption that $G$ is $\mathcal{L}_s$-free. 

Suppose $s\leq 2k-1$. By the fact that   $K_{\frac{s-1}{2}}\vee \overline{K}_{n-\frac{s-1}{2}}$ is $\{K_{k+1},\mathcal{L}_s\}$-free, we have  $G= K_{\frac{s-1}{2}}\vee \overline{K}_{n-\frac{s-1}{2}}$. 

Suppose $s\geq 2k+1$. Since $\ell \leq k-1$, $K_{b_1,\cdots,b_{\ell},n-\frac{s-1}{2}}$ is $K_{k+1}$-free. Furthermore, $K_{b_1,\cdots,b_{\ell},n-\frac{s-1}{2}}$ is $\mathcal{L}_s$-free as $K_{\frac{s-1}{2}}\vee \overline{K}_{n-\frac{s-1}{2}}$ is $\mathcal{L}_s$-free and $K_{b_1,\cdots,b_{\ell},n-\frac{s-1}{2}}\subseteq K_{\frac{s-1}{2}}\vee \overline{K}_{n-\frac{s-1}{2}}$.  Together with the fact that spectral radius of a graph does not decrease by adding an edge, so we have $G=K_{b_1,\cdots,b_{\ell},n-\frac{s-1}{2}}$ for some $b_1,\cdots, b_{\ell}$.
We claim  $\ell =k-1$. Otherwise let $G_2$ be the graph obtained from $G$ by adding new edges in $G[R]$ to make it be a complete $(k-1)$-partite graph. Clearly, $G_2$ is $K_{k+1}$-free. Then $G_2$ is $\{K_{k+1},\mathcal{L}_s\}$-free. While $\rho (G_2)>\rho (G)$ holds.  Therefore $G[R]$ is a complete $(k-1)$-partite graph.  Then  we have $G=K_{b_1,\cdots, b_{k-1},n-\frac{s-1}{2}}$.
Moreover Lemma \ref{balance} implies $\lvert b_i-b_j\rvert \leq 1$ for any $1\leq i<j\leq k-1$. Together with the fact $\sum_{i=1}^{k-1} b_i=\frac{s-1}{2}$,  $G=T_{\frac{s-1}{2},\,k-1}\vee \overline{K}_{n-\frac{s-1}{2}}$ holds. 
 
\vspace{0.8em}\noindent
\textbf{Case 2.} $s$ is even.

In this case, we will prove there is at most one edge in $G[\overline{R}]$. Suppose to the contrary that  there are two edges, say $u_1u_2$, $v_1v_2$  in $G[\overline{R}]$ (if $P_3\subseteq G[\overline{R}]$, let $u_2=v_1$). Since $\lvert W\rvert >s+1$ and $G[R,W]$ is a complete bipartite graph, $G[(R\cup W)\setminus\{u_1,u_2,v_1,v_2\}]$ contains a copy of $\frac{s-2}{2}P_{3}$. Then $G$ contains a copy of $\frac{s}{2}P_3$ or $\left(2P_2\cup \frac{s-2}{2}P_3\right)$, which contradicts to the assumption that $G$ is $\mathcal{L}_s$-free.

 Suppose $s\leq 2k-2$. Since  $K_{\frac{s}{2}-1}\vee \left(P_2\cup \overline{K}_{n-\frac{s}{2}-1}\right)$ is $\{K_{k+1},\mathcal{L}_s\}$-free, we have  $G=K_{\frac{s}{2}-1}\vee \left(P_2\cup \overline{K}_{n-\frac{s}{2}-1}\right)$.

Now we suppose $s\geq 2k$.
We claim $\ell =k-1$. Suppose to the contrary  $\ell \leq k-2$. As $K_{b_1,\cdots,b_{\ell}}\vee \left(P_2\cup \overline{K}_{n-\frac{s}{2}-1}\right)$ is $\{K_{k+1},\mathcal{L}_s\}$-free, we have 
$G=K_{b_1,\cdots,b_{\ell}}\vee \left(P_2\cup \overline{K}_{n-\frac{s}{2}-1}\right)$.
Let  $u_1u_2$ be the edge in $G[\overline{R}]$ and  $\{v_1, v_2\}\subseteq B_{\ell}$. Let $G_3$ be the graph obtained from $G$ by deleting the edge $u_1v_1$ and adding edge $v_1v_2$. If $G_3$ contains  a copy  $K_{k+1}$, then the $K_{k+1}$ contains the vertex $v_1$. However $G_3[N(v_1)]$ is $K_k$-free as $\ell \leq k-2$. Since $K_{\frac{s}{2}-1}\vee (P_2\cup \overline{K}_{n-\frac{s}{2}-1})$ is $\mathcal{L}_s$-free, then $G_3\subseteq K_{\frac{s}{2}-1}\vee (P_2\cup \overline{K}_{n-\frac{s}{2}-1})$ is $\mathcal{L}_s$-free. Therefore, $G_3$ is $\{K_{k+1},\mathcal{L}_s\}$-free. As $u_1\in \overline{R}$ and $v_2\in R$, we have $x_{u_1}< \frac{1}{2(h+1)}x_z\leq x_{v_2}$. Then by Rayleigh-Ritz Theorem we have
\begin{equation*}
	\rho(G_3)-\rho(G)\geq \textbf{x}^T(A(G_3)-A(G))\textbf{x}=2x_{v_1}x_{v_2}-2x_{v_1}x_{u_1}>0.
\end{equation*}
This contradiction implies  $\ell =k-1$.
Now we will prove $G[\overline{R}]=\overline{K}_{n-\frac{s}{2}+1}$.
If there is an edge $u_1u_2$ in $G[\overline{R}]$, without loss of generality, we may suppose $u_1$ is not adjacent to any vertex of   $B_{k-1}$  as $G$ is $K_{k+1}$-free. Let $G_4$ be the graph obtained from $G$ by deleting the edge $u_1u_2$ and
adding edges between $u_1$ and $B_{k-1}$. Then $G_4\subseteq K_{b_1,\cdots,b_{k-1}}\vee \overline{K}_{n-\frac{s}{2}+1}$   is $\{K_{k+1}, \mathcal{L}_s\}$-free.  Since $B_{k-1}\subseteq R$ and $u_2\in \overline{R}$,  $x_{u_2}<\frac{1}{2(h+1)}x_z\leq x_{w}$ holds for each $w\in B_{k-1}$. Then by Rayleigh-Ritz Theorem we have
\begin{equation*}
	\rho(G_4)-\rho(G)\geq \textbf{x}^T(A(G_4)-A(G))\textbf{x}=2x_{u_1}\left(\sum_{w\in B_{k-1}}x_{w}-x_{u_2}\right)>0.
\end{equation*}
This contradiction implies  $G[\overline{R}]=\overline{K}_{n-\frac{s}{2}+1}$. As $K_{b_1,\cdots ,b_{k-1},n-\frac{s}{2}+1}$ is $\{K_{k+1},\mathcal{L}_s\}$-free, we have  $G=K_{b_1,\cdots ,b_{k-1},n-\frac{s}{2}+1}$. Moreover Lemma \ref{balance} implies $\lvert b_i-b_j\rvert \leq 1$ for any $1\leq i<j\leq k-1$. Together with the  fact $\sum_{i=1}^{k-1}b_i=\frac{s}{2}-1$, 
$G=T_{\frac{s}{2}-1,\,k-1}\vee \overline{K}_{n-\frac{s}{2}+1}$ holds.

 \vspace{1em}\noindent
	{\textbf{Declaration}}
	
	The authors have declared that no competing interest exists.
	
\end{document}